\documentclass[reqno,10pt]{amsart}
\usepackage{graphicx,amsfonts,amssymb,amsmath,amsthm,url,amscd,comment}
\usepackage{color}
\usepackage{chngcntr}
\usepackage{enumerate}
\usepackage[usenames,dvipsnames]{xcolor}
\usepackage[normalem]{ulem}
\usepackage{pdfsync}
 \usepackage[hyperfootnotes=false, colorlinks, citecolor=	 RoyalBlue, urlcolor=blue, linkcolor=blue         ]{hyperref}

\theoremstyle{plain} 
\newtheorem{theorem}             {Theorem} 

\theoremstyle{definition}

\theoremstyle{plain} 
\newtheorem{proposition} {Proposition}

\theoremstyle{remark}

\counterwithin*{definition}{subsubsection}
\newtheorem*{definition*}  {Definition}

\newtheorem*{example*}    {Example}

\newtheorem*{remark*}            {Remark}

\newtheoremstyle{itplain} 
    {6pt}                    
    {5pt\topsep}                    
    {\itshape}                   
    {}                           
    {\itshape}                   
    {.}                          
    {5pt plus 1pt minus 1pt}                       
    {}  

\theoremstyle{itplain} 
\newtheorem{lemma}{Lemma}

\newtheorem*{lemma*}{Lemma}

\newtheorem*{corollary*} {Corollary} 

\theoremstyle{remark} 

\newtheorem*{lemmatest*}{Lemma}

\counterwithin*{lemma}{subsubsection}
\counterwithin*{remark}{subsubsection}
\counterwithin*{corollary}{subsubsection}

\usepackage{etoolbox}
\patchcmd{\section}{\scshape}{\bfseries}{}{}
\makeatletter
\renewcommand{\@secnumfont}{\bfseries}
\makeatother

\renewcommand{\Re}{\mathrm{Re}}

\renewcommand{\geq}{\geqslant}
\renewcommand{\leq}{\leqslant}

\numberwithin{equation}{section}

\DeclareMathOperator{\GL}{GL}

\def\eps{\varepsilon}

\def\O{\operatorname{O}}

\author{Roman Holowinsky, Paul D. Nelson}
\address{The Ohio State University, Department of Mathematics, 100 Math Tower, 231 West 18th Ave, Columbus, OH, 43210, USA}
\address{ETH Z{\"u}rich, Department of Mathematics, R{\"a}mistrasse 101, CH-8092, Z{\"u}rich, Switzerland}
\email{holowinsky.1@osu.edu}
\email{paul.nelson@math.ethz.ch}

\date{\today}
\title{Subconvex bounds
  on $\GL_3$ via degeneration to frequency zero}
\hypersetup{
  pdfkeywords={},
  pdfsubject={},
  pdfcreator={Emacs 24.5.1 (Org mode 8.2.10)}}
\begin{document}

\begin{abstract}
  For a fixed cusp form $\pi$ on $\GL_3(\mathbb{Z})$
  and a varying Dirichlet character $\chi$ of prime conductor
  $q$,
  we prove that the subconvex bound
  \[
  L(\pi \otimes \chi, \tfrac{1}{2}) \ll q^{3/4 - \delta}
  \]
  holds for any $\delta < 1/36$.  This improves upon the earlier
  bounds $\delta < 1/1612$ and $\delta < 1/308$ obtained by
  Munshi using his $\GL_2$ variant of the $\delta$-method.  The
  method developed here is more direct.
  We first express $\chi$
  as the degenerate zero-frequency contribution of a carefully
  chosen summation formula {\`a} la Poisson.  After an
  elementary ``amplification'' step exploiting the
  multiplicativity of $\chi$, we then apply
  a sequence of standard  manipulations
  (reciprocity, Voronoi, Cauchy--Schwarz and the
  Weil bound) to bound the contributions of the nonzero
  frequencies and of the dual side of that formula.
\end{abstract}
\maketitle

\setcounter{tocdepth}{1}
\tableofcontents

\section{Introduction}
\label{sec-1}
We consider
the
problem of bounding
$L(\pi \otimes \chi,\tfrac{1}{2})$,
where
\begin{itemize}
\item $\pi$ is a fixed cusp form on $\GL_3(\mathbb{Z})$,
  not necessarily self-dual, and
\item  $\chi$
  traverses a sequence of Dirichlet characters $\chi$ of (say) prime
  conductor $q$ tending off to $\infty$.
\end{itemize}
Munshi \cite{MR3418527} recently 
established the first subconvex bound
in this setting
by showing
that if
$\pi$ satisfies
the Ramanujan--Selberg
conjecture,
then
for any fixed $\delta < 1/1612$,
the estimate
\begin{equation}\label{eqn:subconvex-bound}
  |L(\pi \otimes \chi,\tfrac{1}{2})|
  \leq C q^{3/4 - \delta}
\end{equation}
holds
for some positive quantity $C$ that may
depend upon $\delta$ and $\pi$, but
not upon $\chi$.
In the preprint \cite{2016arXiv160408000M},
he
improves the exponent range to $\delta < 1/308$
and
removes the Ramanujan--Selberg assumption.

A striking feature
of his work is the introduction of a novel ``$\GL_2$
$\delta$-symbol method,'' whereby one detects an equality of
integers $n_1 = n_2$ by averaging
several instances of the Petersson trace formula.
We summarize this approach in Appendix \ref{sec:discussion},
referring to \cite{MR3418527} and \cite{2016arXiv160408000M} for
details,
to \cite{2017arXiv170905615M} and \cite{2017arXiv171002354M} for other recent applications of the $\GL_2$
$\delta$-symbol method, and to \cite[\S5.5]{Iw97} for general discussion of the spectral decomposition of the $\delta$-symbol.

It is natural to ask about the true strength of the $\GL_2$
$\delta$-symbol method.
How does it compare to the classical $\delta$-symbol
method of Duke--Friedlander--Iwaniec \cite{DFI93} and
Heath-Brown \cite{MR1421949}?  For which problems does one fail
and the other succeed? For which problems are the two methods
``identical'' or ``equivalent''? Can the $\GL_2$ $\delta$-symbol
method be simplified or removed in certain applications?

In pondering such questions, we were able to better understand
the arithmetical structure and mechanisms underlying Munshi's argument and construct a more direct proof of the following quantitative strengthening of Munshi's bound.
\begin{theorem}\label{thm:main}
  The subconvex bound \eqref{eqn:subconvex-bound}
  holds for any $\delta < \delta_0 := 1/36$.
\end{theorem}

The proof is surprisingly short compared to earlier proofs of
related estimates.
Indeed, we regard the primary novelty of
this work as not in the numerical improvement of the exponent
$\delta$ but rather in the drastic simplification obtained for
the proof of any subconvex bound \eqref{eqn:subconvex-bound}.

Our point of departure is a formula (see
\S\ref{sec:formula-chi}), derived via Poisson summation, that
expresses $\chi$ in terms of additive characters and twisted
Kloosterman sums.  We insert this into an approximate functional
equation for $L(\pi \otimes \chi, 1/2)$.  After an elementary
``amplification'' step exploiting the multiplicativity of
$\chi$, we then conclude via standard manipulations.  We discuss
in Appendix \ref{sec:discussion} how we arrived at this approach
through a careful study of Munshi's arguments.

We hope that the technique described here may be applied to many
other problems.  For instance, it seems natural to ask whether
it allows a simplification or generalization of the
arguments of \cite{2017arXiv170905615M}
for bounding symmetric square $L$-functions.

The works
\cite{MR2753605,MR2975240,
  MR3369905,2017arXiv170304424N,MR3357122,2016arXiv160509487H,
  MR3766865,2017arXiv170500804S,2018arXiv180310973S,2018arXiv180506026Q}
bound twisted $L$-functions on $\GL_3$ in other aspects.  In the
preprint \cite{2018arXiv180205111L}, Yongxiao Lin has
generalized our method to incorporate the $t$-aspect.  The
preprint \cite{2018arXiv180300542A} applies a simpler
technique to the corresponding problem for $\GL_2$.

\section{Preliminaries}
\label{sec-2}
\subsection{Asymptotic notation}
\label{sec:asymptotic-notation}
We work throughout this article
with a 
cusp form $\pi$  on $\GL_3(\mathbb{Z})$
and a
sequence of primitive Dirichlet characters $\chi_{\mathfrak{j}}$
to prime moduli $q_\mathfrak{j}$, indexed
by $\mathfrak{j} \in \mathbb{Z}_{\geq 1}$, with $q_\mathfrak{j} \rightarrow \infty$.
To simplify notation, we drop the subscripts and write simply
$\chi := \chi_\mathfrak{j}$ and $q := q_\mathfrak{j}$.
Our convention is that any
object (number, set, function, ...)  considered below
may depend implicitly upon $\mathfrak{j}$
unless we designate it as \emph{fixed};
it must then be independent of $\mathfrak{j}$.
Thus $\pi$ is understood as fixed, while $\chi$ is not.
All assertions are to be understood
as holding
after possibly passing to
some subsequence $q_{\mathbf{j}_k}$
of the original sequence $q_\mathbf{j}$,
and in particular,
for $\mathfrak{j}$ sufficiently large.

We define standard asymptotic notation accordingly: $A = \O(B)$
or $A \ll B$ or $B \gg A$ means that $|A| \leq c |B|$ for some
fixed $c \geq 0$, while $A = o(B)$ means $|A| \leq c |B|$ for
every fixed $c > 0$ (for $\mathfrak{j}$ large enough, by convention).
We write $A \asymp B$ for $A \ll B \ll A$.
We write $A = \O(q^{-\infty})$
to denote that $A = \O(q^{-c})$ for each fixed $c \geq 0$.
Less standardly,
we write $A \prec B$ or
$B \succ A$ as shorthand for $A \ll q^{o(1)} B$,
or equivalently,
$|A| \leq q^{o(1)} |B|$.
Our goal is then to show that
\begin{equation}\label{eqn:goal-after-notatino}
  L(\pi \otimes \chi, \tfrac{1}{2})
  \prec q^{3/4-\delta_0}.
\end{equation}

We say that 
$V \in C_c^\infty(\mathbb{R}^\times_+)$
is \emph{inert}
if it satisfies the support condition
\[
V(x) \neq 0 \implies x \asymp 1
\]
and the value and derivative bounds
\[
(x \partial_x)^j V(x) \prec 1 \text{ for each fixed } j \geq 0.
\]

\subsection{General notation}
\label{sec-2-2}
We write $e(x) := e^{2 \pi i x}$,
and denote by $\sum_n$ a sum over integers $n$.
Let $c \in \mathbb{Z}_{\geq 1}$.
We write $\sum_{a(c)}$ and $\sum_{a(c)^*}$
to denote sums over $a \in \mathbb{Z}/c$
and $a \in (\mathbb{Z}/c)^*$, respectively.
We denote the inverse of  $x \in (\mathbb{Z}/c)^*$
by $x^{-1}$ or $1/x$.
We denote by $e_c : \mathbb{Z}/c \rightarrow
\mathbb{C}^{\times}$
the additive character
given by $e_c(a) := e^{2 \pi i a/c}$,
by
$S(a,b;c) := \sum_{x(c)^*}
e_c(a x + b x^{-1})$
the Kloosterman sum,
by
$K_c(a) := c^{-1/2}
S(a,1;c)$ the normalized Kloosterman sum,
by $S_\chi(a,b;q) := \sum_{x(q)^*} \chi(x) e_q(a x + b x^{-1})$
the twisted Kloosterman sum,
and
by
$\eps(\overline{\chi}) :=
q^{-1/2} \sum_{a(q)^*} \overline{\chi }(a) e_q(a)$
the normalized Gauss sum
(of magnitude one).

We define the Fourier coefficients $\lambda(m,n)$ of $\pi$
as in \cite{MR3468028},
so that
$L(\pi \otimes \chi,s) = \sum_{n \in \mathbb{Z}_{\geq 1}}
\lambda(1,n) \chi(n) n^{-s}$ for complex numbers $s$ with large
enough real part,
and $\lambda(n,m) = \overline{\lambda(m,n)}$.

For a condition $C$,
we define $1_{C}$ to be $1$ if $C$ holds
and $0$ otherwise.
For instance, $1_{a = b}$ is $1$ if $a=b$
and $0$ if $a \neq b$.

We denote by
$\hat{V}(\xi) := \int_{x \in \mathbb{R}} V(x) e(-\xi x) \, d x$
the Fourier transform
of a Schwartz function
$V$ on $\mathbb{R}$.

For a pair of integers $a,b$,
we denote by $(a,b)$ and $[a,b]$ their the greatest common divisor
and least common multiple, respectively.

\subsection{Voronoi summation formula\label{sec:voronoi}}
\label{sec-2-3}
By \cite{MR2247965} (cf. \cite[\S4]{2017arXiv170601245B}
for the formulation used here),
we have for $V \in C_c^\infty(\mathbb{R}^\times_+)$,
$m, c \in \mathbb{Z}_{\geq 1}, a \in (\mathbb{Z}/c)^*$, and $X > 0$
that
\begin{equation}\label{eq:voronoi-general}
  \sum_{n}
  V(\frac{n}{X})
  \lambda(m,n)
  e_c(a n)
  =
  c
  \sum _{\substack{\pm, n \\ d \mid c m}}
  \mathcal{I}_{\pm}V \left( \frac{n d^2}{c^3 m/X} \right)
  \frac{\lambda(n,d)}{n d}
  S (\frac{m}{a},
  \pm n;
  \frac{m c}{d}),
\end{equation}
for integral transforms
$V \mapsto \mathcal{I}_{\pm} V \in C^\infty(\mathbb{R}^\times_+)$
of the shape
\[
\mathcal{I}_{\pm}V(x)
= \int_{\Re(s) = 1}
x^{-s} \mathcal{G}^{\pm}(s+1) (\int_{y \in \mathbb{R}_+^\times} V(y) y^{-s} \,
\frac{d y}{y} )
\, \frac{d s}{2 \pi i },
\]
where $\mathcal{G}^{\pm}$ is meromorphic on $\mathbb{C}$
and 
holomorphic in the domain $\Re(s) > 5/14$,
where it
satisfies
$\mathcal{G}^{\pm}(s) \ll (1 + |s|)^{\O(1)}$
for fixed $\Re(s)$.
(The indices $n$ and $d$ in
\eqref{eq:voronoi-general} are implicitly
restricted to be positive integers.)
Set $\theta = 5/14 + \eps$ for some sufficiently small
fixed
$\eps > 0$.
By shifting the contour to $\Re(s) = \theta-1$
and to $\Re(s) = A$,
we see that
if $V$ is inert,
then
\begin{equation}\label{eq:voronoi-transform-bound}
  (x \partial_x)^j \mathcal{I}_{\pm} V(x) \ll
  \min(x^{1-\theta}, x^{-A})
\end{equation}
for all fixed $j,A \geq 0$.

In the special case $m = 1$,
we have
$S (m/a, \pm n; m c/d)
= (c/d)^{1/2} K_{c/d}(\pm n/a)$,
and so
\[
\sum_{n}
V(\frac{n}{X})
\lambda(1,n)
e_c(a n)
=
c^{3/2}
\sum _{\substack{\pm, n \\ d \mid c}}
\mathcal{I}_{\pm}V \left( \frac{n d^2}{c^3/X} \right)
\frac{\lambda(n,d)}{n d^{3/2}}
K_{c/d}(\frac{\pm n}{a}).
\]

\subsection{Rankin--Selberg bounds}
By \cite{MR1876443},
we have for each fixed $\eps > 0$
and all $X \geq 1$ that
\begin{equation}\label{eq:rankin-selberg}
  \sum_{n \leq X}
  |\lambda(n,1)|^2
  =
  \sum_{n \leq X}
  |\lambda(1,n)|^2
  \leq 
  \sum_{m,n: m^2 n \leq X}
  |\lambda(m,n)|^2
  \ll
  X^{1+\eps}.
\end{equation}
Using the Hecke relations
as in the proof of \cite[Lem 2]{2016arXiv160408000M},
we deduce that
for all $M, N \geq 1$,
\begin{equation}\label{eq:rankin-selberg-2}
  \sum_{m \leq M, n \leq N}
  |\lambda(m,n)|^2
  \ll
  (M N)^{1+\eps}.
\end{equation}
(Indeed, we may
reduce to considering
the dyadic sums
over $M/2 < m \leq M, N/2 < n \leq N$,
and then to establishing
that $\sum_{X/8 < m^2 n \leq X} m |\lambda(m,n)|^2 \ll X^{1+\eps}$,
which is shown in \emph{loc. cit.})

\section{Division of the proof}
\label{sec-3}

\subsection{Approximate functional equation}\label{sec:appr-funct-equat}
Recall our main goal \eqref{eqn:goal-after-notatino}.
By \cite[\S5.2]{MR2061214},
we may write
\[
L(\pi \otimes \chi, \tfrac{1}{2}) = \sum_n \frac{\lambda(1,n) \chi(n)}{\sqrt{n}}
V_1 (\frac{n}{q^{3/2}}) + \eta \sum_n \frac{\overline{\lambda(1,n)
    \chi(n)}}{\sqrt{n}} V_2 (\frac{n}{q^{3/2}}),
\]
for some $\eta \in \mathbb{C}$ with $|\eta| = 1$ and some smooth
functions
$V_1, V_2 : \mathbb{R}^\times_+ \rightarrow \mathbb{C}$
satisfying $(x \partial_x)^j V_i(x) \ll \min(1, x^{-A})$ for all
fixed $j, A \in \mathbb{Z}_{\geq 0}$.
By a smooth dyadic partition of unity
and the Rankin--Selberg estimate \eqref{eq:rankin-selberg},
it will suffice to show for each $0 < N \prec q^{3/2}$
and each inert $V \in C_c^\infty(\mathbb{R}^\times_+)$
that
the normalized sum
\[
\Sigma := \sum_{n} \frac{V (n/N)}{N}
\lambda(1,n) \chi(n)
\]
satisfies the estimate
\begin{equation}\label{eq:Sigma-goal}
  \Sigma \prec N^{-1/2} q^{3/4 - \delta_0}.
\end{equation}
By further application of \eqref{eq:rankin-selberg},
we may and shall assume further that
\begin{equation}\label{eq:ranges-for-N}
  q^{3/2 - 2 \delta_0} \leq  N \prec q^{3/2}.
\end{equation}

The proof of \eqref{eq:Sigma-goal} will involve positive parameters
$R,S,T$
satisfying
\begin{equation}\label{eq:RST-basic-bounds}
  q^\eps \ll R,S,T \ll q^{1-\eps} \text{ for some fixed } \eps > 0.
\end{equation}
Thus every integer in 
$[R, 2 R] \cup [S, 2 S] \cup [T, 2 T]$ is coprime to $q$.

\subsection{A formula for $\chi$}\label{sec:formula-chi}
Fix a smooth function $W$ on $\mathbb{R}$
supported in the interval $[1,2]$ with $\int W(x) \, d x = 1$.
Then $\hat{W}(0) = 1$.
Observe that $1/r \in \mathbb{Z}/q$ is defined
for all integers $r$ for which $W(r/R) \neq 0$.
Set \[H := q/R.\]
By Poisson summation,
we have
\begin{equation}\label{eq:poisson-for-chi-formula}
  \frac{\sqrt{q}}{R}
  \sum_{r} {W} (\frac{r}{R})
  \chi(r)
  e_q (\frac{u}{r})
  = 
  \sum_h
  \hat{W} (\frac{h}{H})
  \frac{1}{\sqrt{q}}
  \underbrace{
    \sum_{r(q)^*}
    \chi(r)
    e_q (\frac{u}{r})
    e_q(h r)
  }_{
    = S_\chi(h,u;q)
  }.
\end{equation}
For $h \equiv 0 \pmod{q}$,
we have
$S_\chi(h,u;q) = \sqrt{q} \eps(\overline{\chi }) \chi(u)$.
Setting
$\alpha_r :=    \eps(\overline{\chi })^{-1} R^{-1} W(r/R) \chi(r)$,
we deduce by rearranging \eqref{eq:poisson-for-chi-formula}
that
\begin{equation}\label{eq:formula-chi}
  \chi(u)
  =
  q^{1/2}
  \sum_r
  \alpha_r
  e_q (\frac{u}{r})
  -
  \eps(\overline{\chi })^{-1}
  \sum_{h \neq 0}
  \hat{W} (\frac{h}{H})
  \frac{S_\chi(h,u;q)}{\sqrt{q}}.
\end{equation}
The properties of the sequence
$\alpha$ to be used in what follows
are that it is supported on $[R,2 R]$
and satisfies
the estimates $\alpha_r \prec R^{-1}$ and $\sum_r |\alpha_r| \asymp 1$.

\subsection{``Amplification''}
We choose sequences of complex numbers $\beta_s$
and $\gamma_t$
supported on (say) primes in the intervals $[S, 2 S]$ and $[T, 2 T]$,
respectively,
so that
\begin{equation}\label{eqn:beta-gamma-defining-props}
  \beta_s \prec S^{-1}, \quad \gamma_t \prec T^{-1},
  \quad \sum_s \beta_s \overline{\chi}(s) = \sum_t \gamma_t \chi(t) =
  1.
\end{equation}
Then
\begin{equation}\label{eq:Sigma-averaged-S-T}
  \Sigma =
  \sum_{n,s,t}
  \frac{V (n/N)}{N}
  \lambda(1,n)
  \beta_s \gamma_t \chi (\frac{t n}{s}).
\end{equation}
The properties of $\beta_s$ and $\gamma_t$
just enunciated,
rather than an explicit choice,
are all that will be used;
one could take, for instance
$\beta_s := 
\chi(s)
|\mathcal{P} \cap [S, 2 S]|^{-1} 1_{s \in
  \mathcal{P} \cap [S,2 S]}$,
where $\mathcal{P}$ denotes the set of primes,
and similarly for $\gamma_t$.

\subsection{A formula for $\Sigma$}\label{sec:formula-sigma}
Substituting \eqref{eq:formula-chi}
with $u = t n / s$
into \eqref{eq:Sigma-averaged-S-T}
gives
$\Sigma = \mathcal{F} - \eps(\overline{\chi })^{-1} \mathcal{O}$,
where
\[
\mathcal{F} =
q^{1/2}
  \sum _{r ,s,t}
  \alpha_r \beta_s \gamma_t
  \sum_n
  \frac{V (n/N)}{N}
  \lambda(1,n)
  e_q (\frac{t n}{r s}),
\]
\[
\mathcal{O} = \sum_n
\frac{V (n/N)}{N}
\lambda(1,n)
\sum _{s,t } \beta_s \gamma_t
\sum _{h \neq 0} \hat{W} (\frac{h}{H})
\frac{S_\chi(h,t n/s;q)}{\sqrt{q}}.
\]

\subsection{Main estimates}
We prove these in the next two sections.
\begin{proposition}\label{prop:F}
  Assume that
  \begin{equation}\label{eqn:cond-for-recip}
    q R S \succ T N.
  \end{equation}
  Then
  \begin{equation}\label{eqn:claimed-bound-F}
    |\mathcal{F}|^2 \prec
    \frac{q}{N}
    \left( \frac{q R S}{T N} \right)^3
    + 
    q \frac{(R S)^3 }{N^2}
    \left( \frac{1}{S T}
      + \frac{1 + N/R^2 S}{R^{1/2} S}
    \right).
  \end{equation}
\end{proposition}
\begin{remark*}
  As explained in the remark of
  \S\ref{sec:F-est-disc-some-noise}, the first term on the RHS
  of \eqref{eqn:claimed-bound-F} is unnecessary.
  Including it simplifies slightly our proofs without affecting our final estimates.
\end{remark*}
\begin{proposition}\label{prop:O}
  Assume that
  \begin{equation}\label{eq:ST-smaller-R}
    S T \leq q^{-\eps} R
  \end{equation}
  for some fixed $\eps > 0$.
  Then
  \begin{equation}\label{eq:bound-for-O}
    |\mathcal{O}|^2 \prec
    H^2 \frac{1}{S T H}.
  \end{equation}
\end{proposition}

\subsection{Optimization}
Our goal reduces to establishing
that
$\mathcal{F}, \mathcal{O}  \prec N^{-1/2} q^{3/4 - \delta_0}$.
(By comparison, we note
the trivial bounds
$\mathcal{F} \prec q^{1/2}$ and
$\mathcal{O} \prec H$.)
We achieve this
by applying the above estimates with
\[
R :=
\frac{T N }{q S},
\quad
S := q^{2/18},
\quad
T := q^{5/18}.
\]
Then \eqref{eqn:cond-for-recip}
is clear, while \eqref{eq:ST-smaller-R}
follows from \eqref{eq:ranges-for-N}.
The required bound for $\mathcal{O}$
follows readily from \eqref{eq:bound-for-O}.
We now deduce the required bound for $\mathcal{F}$.
Note 
that the first term on the RHS
of \eqref{eqn:claimed-bound-F}
is acceptable thanks to our choice of $R$.
Note also
from \eqref{eq:ranges-for-N}
that $q S T \leq N \prec q^{3/2}$;
from our choice of $R$, it follows that
$1/S T \gg (1 + N /R^2 S)/ R^{1/2} S$.
The bound
for $|\mathcal{F}|^2$
then readily simplifies to
$|\mathcal{F}|^2
\prec q^{-2 \delta_0}
\prec N^{-1}  q^{3/2 - 2 \delta_0}$.
(By solving a linear programming problem,
we see moreover that these choices give the optimal bound
for $L(\pi \otimes \chi, 1/2)$
derivable from the above propositions.)

\section{Estimates for $\mathcal{F}$}
\label{sec-3-5}
\label{sec:F-est}
We now prove Proposition \ref{prop:F}.

\subsection{Reciprocity}
\label{sec-3-5-1}
Our assumption \eqref{eqn:cond-for-recip} implies that
for
all $r,s,t$ with $\alpha_r \beta_s \gamma_t \neq 0$,
the function
$V_{r,s,t}' (x)
:=
V(x)
e(t N x/q r s)$
is inert.
By the Chinese remainder theorem,
we have
$e_q (t n /r s)
=
e_{q r s}
(t n)
e_{r s}(- {t n}/{q})$ for $(r s, q) = 1$.
We may thus rewrite
\[
\mathcal{F} = 
\sum _{r ,s,t}
\alpha_r \beta_s \gamma_t
\mathcal{S}(r,s,t),
\]
where
\[
\mathcal{S}(r,s,t)
:=
q^{1/2}
\sum_n
\frac{V_{r,s,t}' (n/N)}{N}
\lambda(1,n)
e_{r s}(- \frac{t n}{q}).
\]
\subsection{Voronoi}
We introduce the notation
\[
c := c(r s, t) := \frac{r s}{(r s, t)},
\quad 
a := a(r s, t) :=
\frac{-t}{(r s, t)},
\]
so that
$e_{r s} (- t n/q)
=
e_c (a n/q)$ and $(a,c) = 1$.
Applying Voronoi summation
(\S\ref{sec:voronoi}),
we obtain
\[
\mathcal{S}(r,s,t)
=
\frac{ q^{1/2} c^{3/2}  }{N}
\sum _{\substack{\pm, n \\ d \mid c}}
V_{\pm, r,s,t}'' \left( \frac{n d^2}{c^3/N} \right)
\frac{\lambda(n,d)}{n d^{3/2}}
K_{c/d}(\frac{\pm q n}{a})
\]
for some smooth functions
$V_{\pm, r,s,t}''$ 
satisfying
$(x \partial_x)^j V_{\pm, r,s,t}''(x)
\prec
\min(x^{1-\theta}, x^{-A})$
for fixed $j, A \in \mathbb{Z}_{\geq 0}$.

\subsection{Cleaning up}\label{sec:F-est-disc-some-noise}
The Weil bound,
the Rankin--Selberg bound
\eqref{eq:rankin-selberg}
and the condition $N \prec q^{3/2}$ give
\begin{equation}\label{eqn:triv-Srst}
  \mathcal{S}(r,s,t) \prec q^{1/2} c^{3/2}/N \prec N^{-1/2} q^{1/2} (q c / N)^{3/2}.
\end{equation}
If $(r s, t) \neq 1$,
then (because $t$ is prime)
$c = r s / t$,
hence by \eqref{eqn:triv-Srst},
\[
\sum_{r,s,t: (r s, t) \neq 1} \alpha_r \beta_s \gamma_t
\mathcal{S}(r,s,t)
\prec
N^{-1/2} q^{1/2}
\left( \frac{q R S }{T N} \right)^{3/2}.
\]
Since
the square of the latter is the first term
on the RHS of \eqref{eqn:claimed-bound-F},
the proof of Proposition \ref{prop:F}
reduces to
that of an adequate bound for
the sum
\[\mathcal{F}_1 := \sum_{r,s,t : (r s, t) = 1} \alpha_r
\beta_s \gamma_t \mathcal{S}(r,s,t).
\]
If $(rs , t) = 1$,
then $c = r s$ and $a = -t$,
hence
\begin{equation}\label{eqn:F1-via-Phi}
  \mathcal{F}_1 = 
  \sum_{\substack{
      \pm, n, r, d \\
    }
  }
  \alpha_r
  \frac{\lambda(n,d)}{\sqrt{n d}}
  \sum _{\substack{
      s, t : \\
      d | r s, (r s, t) = 1
    }
  }
  \beta_s \gamma_t
  \Phi(n,d,r,s,t)
\end{equation}
with
\[
\Phi(n,d,r,s,t)
:=
\frac{q^{1/2} (r s)^{3/2}}{N \sqrt{n} d}
V_{\pm, r,s,t}'' \left( \frac{n d^2}{(r s)^3/N} \right)
K_{ r s / d}( \frac{\mp q n }{ t}).
\]

\begin{remark*}
  With slightly more case-by-case analysis in the arguments to follow,
  one can verify that the reduction performed here to the case
  $(rs , t) = 1$ is unnecessary,
  hence that the bound \eqref{eqn:claimed-bound-F}
  remains valid in the stated generality
  even after deleting the first term on its RHS.
\end{remark*}

\subsection{Cauchy--Schwarz}\label{sec:F-CS}
Let $\eps >0$ be fixed and small.
The rapid decay of $V_{\pm,r,s,t}''$
implies that truncating \eqref{eqn:F1-via-Phi}
to $n d^2 \leq q^\eps (R S)^3/N$
introduces the negligible error $\O(q^{-\infty})$.
By the Rankin--Selberg bound \eqref{eq:rankin-selberg-2},
we have
\[
\sum_{\substack{
    \pm, n, r, d: \\
    n d^2 \leq q^\eps (R S)^3/N
  }
}
|\alpha_r|
\frac{|\lambda(n,d)|^2}{n d}
\prec 1.
\]
It follows by Cauchy--Schwarz
that
\[
|\mathcal{F}_1|^2
\prec
\sum_{\substack{
    \pm, n, r, d
  }
}
|\alpha_r|
\left\lvert
  \sum _{\substack{
      s, t : \\
      d | r s, (r s, t) = 1
    }
  }
  \beta_s \gamma_t
  \Phi(n,d,r,s,t)
\right\rvert^2 + \O(q^{-\infty}).
\]

\subsection{Application of exponential sum bounds}
Opening the square,
expanding the definition of $\Phi$
and wastefully discarding some summation conditions,
we obtain
\begin{equation}\label{eqn:bound-F1-via-C}
  |\mathcal{F}_1|^2
  \prec
  \frac{q (R S)^3}{N^2}
  \sum_{\substack{
      \pm,r,d,s_1,s_2,t_1,t_2: \\
      d|(r s_1,r s_2)
    }
  }
  \frac{|\alpha_r \beta_{s_1} \beta_{s_2} \gamma_{t_1}
    \gamma_{t_2}|}{d^2}
  |\mathcal{C}| + \O(q^{-\infty}),
\end{equation}
where $\mathcal{C}$ is defined
for $(r,s_1,s_2,t_1,t_2)$ in the support of
$\alpha_r \beta_{s_1} \beta_{s_2} \gamma_{t_1} \gamma_{t_2}$
by
\begin{equation}\label{eq:formula-fro-C}
\mathcal{C}
:=
\frac{1}{X}
\sum_n
U (\frac{n}{X})
K_{r s_1/d}
(\frac{\mp q n}{t_1})
\overline{
  K_{r s_2/d}
  (\frac{\mp q n}{t_2})
}
\end{equation}
with
\[
X := \frac{(r s_1)^{3/2} (r s_2)^{3/2}}{d^2 N} \asymp \frac{(R S)^3}{d^2 N}
\]
and
\[
U(x) :=
\frac{1}{x}
V_{\pm,r,s_1,t_1}''\left( \frac{X x d^2}{(r s_1)^3/N} \right)
\overline{
  V_{\pm,r,s_2,t_2}''\left( \frac{X x d^2}{(r s_2)^3/N} \right)
}.
\]
We have
$(x \partial_x)^j U(x)
\prec
\min(x^{1 - 2 \theta }, x^{-A})$ for fixed $j, A \in \mathbb{Z}_{\geq 0}$.
By a smooth dyadic partition of unity,
we may write
\begin{equation}\label{eq:part-unity-W-dotsb}
  U(x)
  = \sum_{Y \in \exp(\mathbb{Z})}
  \min(Y^{1-2\theta}, Y^{-10}) U_Y(\frac{x}{Y}),
\end{equation}
where each function $U_Y$ is inert.
Substituting \eqref{eq:part-unity-W-dotsb}
into \eqref{eq:formula-fro-C}
and applying the incomplete exponential sum estimates
recorded in Appendix \ref{sec:corr-kloost},
we obtain
with
\[
\Delta :=
q
\frac{(r s_2/d)^2 t_2 - (r s_1/d)^2 t_1}{(r s_1/d,r s_2/d)^2}
=
q
\frac{s_2^2 t_2 - s_1^2 t_1}{(s_1,s_2)^2}
\]
that
\[
\mathcal{C} 
\prec
\frac{1}{X}
\sum_{Y \in \exp(\mathbb{Z})}
\min(Y^{1-2 \theta}, Y^{-10})
\left(
  X Y
  \frac{(\Delta, \tfrac{r s_1}{d}, \tfrac{r s_2}{d})^{1/2}}{[\tfrac{r s_1}{d},\tfrac{r s_2}{d}]^{1/2}}
  + [\tfrac{r s_1}{d},\tfrac{r s_2}{d}]^{1/2}
\right).
\]
Since $\theta < 1/2$, the above sum is dominated
by the contribution from $Y = 1$;
estimating that contribution a bit crudely with respect to $d$,
we obtain
\begin{equation}\label{eqn:calC-final-bound}
  \mathcal{C}  \prec d^{1/2} 
\frac{(\Delta, r s_1, r s_2)^{1/2}}{r^{1/2} [s_1,s_2]^{1/2}}
+
\frac{r^{1/2} [s_1,s_2]^{1/2}}{d^{1/2} X}.
\end{equation}
\subsection{Diagonal and off-diagonal}
To state the
estimates to be obtained
shortly,
we introduce
the notation
\[
\mathbb{E}_{r,s_1,s_2,t_1,t_2} :=
\frac{1}{R S^2 T^2}
\sum_{\substack{
    r :
    \\
    R \leq s \leq 2 R
  }
}
\sum_{\substack{
    s_1, s_2: \\
    S \leq s_1,s_2 \leq 2 S
  }
}
\sum_{\substack{
    t_1, t_2: \\
    T \leq t_1,t_2 \leq 2 T
  }
}.
\]
We estimate separately the contribution of each term on the RHS
of \eqref{eqn:calC-final-bound} to $\mathcal{F}_1$ via
\eqref{eqn:bound-F1-via-C}, splitting off the contribution to
the first from terms with $\Delta = 0$.
We obtain in this way that
\[|\mathcal{F}_1|^2 \prec
\frac{q (R S)^3}{N^2}
\sum_{i=0,1,2} \mathcal{B}_i +
\O(q^{-\infty}),\]
where
\begin{align*}
  \mathcal{B}_0 &:=
                  \mathbb{E}_{r,s_1,s_2,t_1,t_2}
                  1_{\Delta = 0}
                  \frac{(s_1, s_2)^{1/2}}{[s_1,s_2]^{1/2}},
  \\
  \mathcal{B}_1 &:=
                  \mathbb{E}_{r,s_1,s_2,t_1,t_2}
                  1_{\Delta \neq 0}
                  \frac{(\Delta , r s_1, r s_2)^{1/2}}{r^{1/2} [s_1,s_2]^{1/2}},
  \\
  \mathcal{B}_2 &:=
                  \frac{N}{(R S)^3}
                  \mathbb{E}_{r,s_1,s_2,t_1,t_2}
                  r^{1/2} [s_1,s_2]^{1/2}.
\end{align*}
(In deriving
the estimate involving $\mathcal{B}_2$,
we used the slightly wasteful
bound
$\frac{1}{d^2} \frac{1}{d^{1/2} X}
\ll \frac{N}{ (R S)^3}$.)
Noting that $\Delta = 0$ iff $s_2^2 t_2 = s_1^2 t_1$,
we verify
using the divisor bound
that
\begin{align*}
  \mathcal{B}_0
  &\prec
    \frac{1}{ S T}, \\
  \mathcal{B}_1 &\prec
                  \frac{1}{R^{1/2} S}, \\
  \mathcal{B}_2
  &\prec
    \frac{N}{(R S)^3}
    R^{1/2} S.
\end{align*}
These estimates combine to give an adequate estimate
for $\mathcal{F}_1$.

\section{Estimates for $\mathcal{O}$}
\label{sec-3-6}
\label{sec:O-est}
We now prove Proposition \ref{prop:O}.
\subsection{Cauchy--Schwarz}
Using again the Rankin--Selberg bound \eqref{eq:rankin-selberg},
we obtain
\[
|\mathcal{O}|^2
\prec
\sum_{n}
\frac{|V(n/N)|^2}{N}
\left\lvert
  \sum_{s, t, h : h \neq 0}
  \beta_s \gamma_t \hat{W} (\frac{h}{H})
  \frac{S_\chi(h, t n/s; q)}{\sqrt{q}}
\right\rvert^2.
\]
\subsection{Elementary exponential sum bounds}
Let $\eps >0$ be fixed but sufficiently small.
Since $q$ is prime
and $R$ satisfies the  lower bound in
\eqref{eq:RST-basic-bounds},
we know that the integers $h$ and $q$
are coprime whenever
$0 \neq |h| \leq q^\eps H$.
By the rapid decay of $\hat{W}$,
we may truncate the $h$-sum
to $|h| \leq q^\eps H$ with negligible
error $\O(q^{-\infty})$.
We then open the square and apply Cauchy--Schwarz,
leading us to  consider
for $s_1,t_1,h_1,s_2,t_2,h_2$
with
\begin{equation}\label{eqn:sth-conditions}
  S \leq s_i \leq 2 S,
  \quad 
  T \leq t_i \leq 2 T,
  \quad 
  0 \neq |h_i| \leq q^\eps H
\end{equation}
the sums 
\begin{equation}\label{eq:}
  \Pi :=
  \sum_n
  \frac{|V(n/N)|^2}{N}
  \frac{S_\chi(h_1, t_1 n/s_1; q)}{\sqrt{q}}
  \overline{\frac{S_\chi(h_2, t_2 n/s_2; q)}{\sqrt{q}}}.
\end{equation}
We apply Poisson summation.
By the lower bound on $N$ in \eqref{eq:ranges-for-N}
and the assumption $\delta_0  = 1/36 < 1/4$,
we have $N \gg q^{1+\eps}$ for some fixed $\eps > 0$.
Thus only the zero frequency $\xi = 0$
after Poisson
contributes non-negligibly,
and so
$\Pi
\prec
q^{-1} \Pi_0 + \O(q^{-\infty})$
with
\[\Pi_0
:=
\sum_{n (q)}
\frac{S_\chi(h_1, t_1 n/s_1; q)}{\sqrt{q}}
\overline{\frac{S_\chi(h_2, t_2 n/s_2; q)}{\sqrt{q}}}.
\]
Opening the Kloosterman sums
and executing the $n$-sum
gives
\[
\Pi_0 = \sum_{x,y(q)^*}
1 _{t_1 / s_1 x = t_2/s_2 y}
\chi(x/y)
e_q(h_1 x - h_2 y).
\]
Our assumptions imply that
the quantities $s_i, t_i, h_i$
are all coprime to $q$,
so after a change of variables
we arrive
at
\[
|\Pi_0|
=
|\sum_{x(q)^*} e_q((s_1 t_2 h_1 - s_2 t_1 h_2) x)|
\leq
(t_1 s_2 h_2 - t_2 s_1 h_1, q).
\]
\subsection{Diagonal vs. off-diagonal}
We have shown thus far that
\[
|\mathcal{O}|^2 \prec
H^2
\frac{1}{(S T H)^2}
\sum_{s_1,t_1,h_1,s_2,t_2,h_2}
q^{-1}
(t_1 s_2 h_2 - t_2 s_1 h_1, q)
+
\O(q^{-\infty}),
\]
where the sum is restricted
by the condition \eqref{eqn:sth-conditions}.
By our assumption \eqref{eq:ST-smaller-R},
the quantities $t_1 s_2 h_2$ and $t_2 s_1 h_1$
are congruent modulo $q$ precisely when they are equal.
By the divisor bound, the number of tuples
for which $t_1 s_2 h_2 = t_2 s_1 h_1$
is
$\prec q^{2 \eps } S T H$.
Since $\eps > 0$ was arbitrary,
we obtain
\begin{equation}\label{eq:}
  |\mathcal{O}|^2
  \prec
  H^2
  \left( \frac{1}{S T H}
    + \frac{1}{q}
  \right).
\end{equation}
By another application of our assumption
\eqref{eq:ST-smaller-R},
the first term in the latter bound dominates,
giving the required bound for $\mathcal{O}$.

The proof of our main result (Theorem \ref{thm:main})
is now complete.

\appendix
\section{Correlations of Kloosterman sums\label{sec:corr-kloost}}
\label{sec-2-4-1}
The estimates recorded here are unsurprising, but we were unable
to find references containing all cases that we require
(compare with e.g. \cite{MR3294387, MR3377053,
  2016arXiv160408000M}).
\begin{lemma}
  Let $s$ be a natural number.
  Let $a,b,c,d \in \mathbb{Z}/s$
  be congruence classes
  for which
  $(d,s) = 1$.
  For each prime $p \mid s$,
  let $\mathcal{X}_0(p) \subseteq \mathbb{Z}/p$
  be a subset of cardinality $p - \O(1)$.
  Let $\mathcal{X}$ denote the set of elements
  $x \in \mathbb{Z}/s$
  for which
  \begin{itemize}
  \item the class of $x$ modulo $p$ belongs
    to $\mathcal{X}_0(p)$ for each $p \mid s$, and
  \item $(c x + d,s) = 1$.
  \end{itemize}
  Define
  $\phi : \mathcal{X} \rightarrow \mathbb{Z}/s$
  by
  \[
    \phi(x) := x \frac{a x + b}{c x + d}.
  \]
  Then the exponential sum
  $\Sigma := s^{-1} \sum_{x \in \mathcal{X}} e_s(\phi(x))$
  satisfies
  \[
    |\Sigma|
    \leq
    2^{\O(\omega(s))}
    \frac{(a,b,s)}{s^{1/2} (a,s)^{1/2}},
    \]
    where $\omega(s)$ denotes the number of prime divisors
    of $s$, without multiplicity.
\end{lemma}
\begin{proof}
  We may assume that $s = p^n$ for some prime $p$.  For $n=0$,
  there is nothing to show.  For $n=1$, we appeal either to the
  Weil bound, to bounds for Ramanujan sums, or to the trivial
  bound according as $(a,p) = 1$, or $(a,p) = p$ and
  $(a,b,p) = 1$, or $(a,b,p) = p$.  We treat the remaining cases
  by induction on $n \geq 2$.
  If $(a,b,p) > 1$, then
  the conclusion follows by our inductive hypothesis applied to
  $s/p, a/p, b/p, c, d$.  We may thus assume that $(a,b,p) = 1$.
  A short calculation gives the identities
  of rational functions
  \begin{equation}\label{eq:2nd-deriv-test}
    \phi '(x)
    = \frac{a c x^2 + 2 a d x + b d}{(c x + d)^2},
    \quad 
    \phi ''(x)
    =  \frac{2 a + 2 c \phi '(x)}{c x + d}.
  \end{equation}
  Write $n = 2 \alpha$ or $2 \alpha + 1$,
  and set
  $\mathcal{R} := \{x \in \mathcal{X}/p^\alpha : \phi '(x)
  \equiv 0 \pod{p^\alpha}\}$.
  Then by $p$-adic stationary phase \cite[\S12.3]{MR2061214},
  \[
    \Sigma \ll
    s^{-1/2}
    \sum_{x \in \mathcal{R}}
    (\phi ''(x), p)^{1/2}.
  \]
  If $(a,p) > 1$, then
  $(b,p) = 1$
  and
  $\phi '(x) \equiv b d / (c x + d)^2 \pod{p}$,
  so $(\phi'(x), p) = 1$.
  Thus
  $\mathcal{R} = \emptyset$ and $\Sigma = 0$.  
  Assume otherwise
  that $(a,p) = 1$.
  For $x \in \mathcal{R}$,
  we have
  $\phi''(x) \equiv 2 a / (c x + d) \pod{p^\alpha}$,
  so that
  \begin{equation}\label{eq:2nd-deriv-woo}
    x \in \mathcal{R} \implies
    (\phi ''(x), p) = (2 a, p^\alpha) = (2,p^\alpha) \ll 1.
  \end{equation}
  Thus $\Sigma \ll s^{-1/2} \# \mathcal{R}$
  and, by Hensel's lemma,
  $\# \mathcal{R} \ll 1$.
  The proof of the required bound is then complete.
\end{proof}

\begin{lemma}
  Let $s_1,s_2$ be natural numbers.
  Let
  $a_1,a_2,b_1,b_2$ be integers with $(b_1,s_1) = (b_2,s_2) =
  1$.
  Set $\ell_i := a_i/b_i \in \mathbb{Z}/s_i$.
  Set
  \[
    \Delta :=
    \frac{s_2^2 b_2 a_1 - s_1^2 b_1 a_2}{(s_1,s_2)^2}.
  \]
  \begin{enumerate}[(i)]
  \item \label{item:complete-exp-sums}
    Let $\xi$ be an integer.
    Set
    \[
      \Sigma  :=
      \frac{1}{[s_1,s_2]}
      \sum _{x([s_1,s_2])}
      K_{s_1}(\ell_1 x)
      \overline{K_{s_2}(\ell_2 x)}
      e _{[s_1,s_2]}(\xi x)
    \]
    Then
    \begin{equation}\label{eq:stronger-corr-kl}
      |\Sigma|
      \leq 2^{\O(\omega([s_1,s_2]))}
      \frac{(\Delta,\xi,s_1,s_2)}{ [s_1,s_2]^{1/2} (\xi, s_1, s_2)^{1/2}}.
    \end{equation}
    In particular,
    \begin{equation}\label{eq:weaker-klooster-corr}
      |\Sigma|
      \leq  2^{\O(\omega([s_1,s_2]))}
      \frac{(\Delta,\xi,s_1,s_2)^{1/2}}{ [s_1,s_2]^{1/2}}.
    \end{equation}
  \item \label{item:incomplete-exp-sums}
    Let $V : \mathbb{R} \rightarrow \mathbb{C}$ be a smooth
    function
    satisfying
    $x^m \partial_x^n V(x) \prec 1$
    for all fixed $m,n \in \mathbb{Z}_{\geq 0}$.
    Let $X > 0$.
    Assume that $s_1, s_2 = \O(q^{\O(1)})$.
    Then
    \begin{equation}\label{eq:incompl-klooster-corr}
      \sum_{n}
      V (\frac{n}{X})
      K_{s_1}(\ell_1 n)
      \overline{K_{s_2}(\ell_2 n)}
      \prec
      X \frac{(\Delta, s_1, s_2)^{1/2}}{[s_1,s_2]^{1/2}}
      +
      [s_1,s_2]^{1/2}.
    \end{equation}
  \end{enumerate}
\end{lemma}
\begin{remark*}
  These estimates are not sharp if either $(a_1,s_1)$ or
  $(a_2,s_2)$ is large, but that case is unimportant for us.
  In fact, we have recorded \eqref{eq:stronger-corr-kl} only for
  completeness; the slightly weaker bound
  \eqref{eq:weaker-klooster-corr} is the relevant one for our
  applications.
  We note finally that $K_{s}$ is real-valued.
\end{remark*}

\begin{proof}
  We begin with \eqref{item:complete-exp-sums}.
  Each side of \eqref{eq:stronger-corr-kl}
  factors naturally as a product over primes,
  so we may assume that $s_i = p^{n_i}$ for some prime $p$.
  By the change of variables $x \mapsto b_1 b_2 x$,
  we may reduce further to the case $b_1 = b_2 = 1$,
  so that $\ell_i = a_i$.
  
  In the case that some $\ell_i$ is divisible by $p$, the
  quantity $K_{s_i}(\ell_i x)$ is independent of $x$, has
  magnitude at most $s_i^{-1/2}$, and vanishes if $n_i > 1$.
  The required estimate then follows in the stronger form
  $\Sigma \ll (s_1 s_2)^{-1/2}$ by opening the other
  Kloosterman sum and executing the sum over $x$.  We will thus
  assume henceforth that $\ell_1$ and $\ell_2$ are coprime to
  $p$.

  Write $w_i := s_i / (s_1,s_2)$,
  so that $w_1 s_2 = s_1 w_2 = [s_1,s_2]$
  and
  $\Delta = w_2^2 b_2 a_1 - w_1^2 b_1 a_2$.
  By opening the Kloosterman sums
  and summing over $x$,
  we obtain
  \begin{equation}\label{eq:sigma-after-parseval}
    \Sigma = 
    \frac{1}{\sqrt{s_1 s_2}}
    \mathop{
      \sum _{x_1(s_1)^*} \, \sum _{x_2(s_2)^*}
    } _{w_1 \ell_2 x_2^{-1} \equiv w_2 \ell_1 x_1^{-1} + \xi \,
      ([s_1,s_2])}
    e_{[s_1,s_2]}
    (w_2 x_1 - w_1 x_2).
  \end{equation}
  Consider first the case $s_1 = s_2 =: s$,
  so that $w_1 = w_2 = 1$ and $\Delta = \ell_1 - \ell_2$ and $[s_1,s_2] = (s_1,s_2) = s$.
  The subscripted identity in \eqref{eq:sigma-after-parseval} then shows
  that $x_2$ is determined uniquely by $x_1 =: x$
  and, after a short calculation, that
  \[
    \Sigma = \frac{1}{s} \sum_{x(s)^*}
    e_s(x \frac{\xi x + \Delta x}{\xi x + \ell_1}).
  \]
  By the previous lemma,
  it follows that
  \[
    \Sigma \ll
    \frac{(\Delta,\xi,s)}{s^{1/2} (\xi,s)^{1/2}},
  \]
  as required.

  Suppose now that $s_1 \neq s_2$.
  Without loss of generality,
  $s_1 < s_2$.
  Then $w_1 = 1$ and $w_2 = s_2/s_1$;
  in particular, $w_2$ is divisible by $p$.
  The summation condition
  in \eqref{eq:sigma-after-parseval}
  shows
  that
  $\Sigma = 0$ unless
  $(\xi,p) = 1$, as we henceforth assume.
  Since $(\ell_1 \ell_2, p) = 1$,
  we have $(\Delta, p) = 1$,
  so our goal is to show that
  $\Sigma \ll s_2^{-1/2}$.
  We introduce the variable
  \[
    y := \xi x_1 + w_2 \ell_1.
  \]
  Then
  \[
    x_1 = \frac{y - w_2 \ell_1}{\xi},
    \quad 
    x_2 =
    \frac{w_1 \ell_2 }{y}
    \frac{y - w_2 \ell_1}{\xi},
  \]
  and as $y$ runs over $(\mathbb{Z}/s_1)^*$,
  the pair $(x_1, x_2)$
  traverses
  the set indicated in \eqref{eq:sigma-after-parseval}.
  A short calculation
  gives
  \[
    w_2 x_1 -
    w_1 x_2
    =
    - \frac{\Delta }{\xi }
    +
    \frac{w_2}{\xi}
    ( y + \frac{\ell_2 \ell_1}{y}),
  \]
  hence
  \[
    \Sigma = 
    \frac{1}{\sqrt{s_1 s_2}}
    e_{s_2} (  - \frac{\Delta }{\xi } )
    \underbrace{
      \sum _{y(s_1)^*} 
      e_{s_1}
      (  \frac{1}{\xi} ( y + \frac{\ell_2 \ell_1}{y}) )
    }_{
      \sqrt{s_1}
      K_{s_1}(\ell_2 \ell_1/\xi^2)
    }.
  \]
  The required conclusion then follows
  from the Weil bound.

  To prove \eqref{item:incomplete-exp-sums},
  we first apply Poisson summation
  to write
  the LHS of \eqref{eq:incompl-klooster-corr}
  as
  \begin{equation}\label{eq:incompl-intermediate}
    X
    \sum_{\xi}
    \hat{V}
    (\frac{\xi }{[s_1,s_2]/X})
    \frac{1}{[s_1,s_2]}
    \sum_{x([s_1,s_2])}
    K_{s_1}(\ell_1 x) \overline{K_{s_2}(\ell_2 x)}
    e_{[s_1,s_2]}(\xi x),
  \end{equation}
  where $\hat{V}$ satisfies
  estimates analogous to those assumed for $V$.
  We then apply \eqref{eq:weaker-klooster-corr}.
  The $\xi = 0$
  term in \eqref{eq:incompl-intermediate}
  then contributes the first
  term on the RHS of \eqref{eq:incompl-klooster-corr},
  while an adequate
  estimate for the remaining terms
  follows from the consequence
  \[
    \sum_{\xi \neq 0}
    |\hat{V}|
    (\frac{\xi }{[s_1,s_2]/X})
    (\Delta,\xi,s_1,s_2)^{1/2}
    \prec [s_1,s_2]/X
  \]
  of the divisor bound.
\end{proof}

\section{Comparison with Munshi's approach}\label{sec:discussion}
We outline Munshi's approach \cite{MR3418527, 2016arXiv160408000M} to the sums $\Sigma$
arising as in 
\S\ref{sec:appr-funct-equat}
after a standard application of the approximate functional
equation, and compare with our own treatment.
For simplicity we focus on the most difficult range $N \approx  q^{3/2}$.

\subsection{Averaged Petersson formula}
Munshi employs the following decomposition of the diagonal symbol:
\begin{align}\nonumber
  \delta(m,n) &=  \frac{1}{B^\star}\sum_{b\in \mathcal{B}} \sum_{\psi(b)} (1-\psi(-1)) \sum_{f \in S_k(b,\psi)} {w_f}^{-1} \overline{\lambda_f(m)} \lambda_f(n) \\ \label{eqn:petersson-formula}
              & \, - 2 \pi i^{-k}  \frac{1}{B^\star}\sum_{b\in \mathcal{B}} \sum_{\psi(b)} (1-\psi(-1)) \sum_{c
                \equiv 0 (b)} \frac{S_\psi(m,n,c)}{c} J_{k-1} (\frac{4 \pi
                \sqrt{m n}}{c}).
\end{align}
Here $\mathcal{B}$ is
a suitable set
of natural numbers, $\psi$ runs over a suitable collection of odd Dirichlet
characters modulo $b \in \mathcal{B}$,
and $B^\star$ denotes the appropriate
normalizing factor.

\subsection{Munshi's initial transformations}
Set $A(n) := \lambda(1,n)$.
Munshi writes\footnote{
  For the sake of comparison,
  we note that Munshi used the notation
  $R,L,P,M$
  corresponding to our
  $R,S,T,q$.
}
\begin{equation}\label{mainobjectofinterest}
  \sum_{n \sim N} A(n) \chi(n) \approx \frac{1}{S} \sum_{s \sim S}
  \sum_{n \sim N} A(n) \sum_{r \sim NS}
  \chi(\frac{r}{s}) \delta(r,ns)
\end{equation}
where $s$ runs over primes of size $S$. Munshi applies \eqref{eqn:petersson-formula} to
$\delta(r,n\ell)$
with $\mathcal{B} = \{t q : t \sim T\}$, where $t$ runs over
primes of size $T$,
and the characters $\psi$ are taken to be trivial modulo $q$.
The use of \eqref{eqn:petersson-formula} produces two main
contributing terms,
$\mathcal{F}^M$
from the sum of Fourier coefficients
and $\mathcal{O}^M$
from the sum of Kloosterman sums,
given roughly by
\begin{equation}\label{Fourier}
  \mathcal{F}^M \approx \frac{1}{T^2S}\sum_{s} \sum_{t} \sum_{\psi(t)} \sum_{n \sim N} \sum_{r \sim NS} A(n) \chi(\frac{r}{s}) \sum_{f \in S_k(t q, \psi)} {\omega_f}^{-1} \overline{\lambda_f(r)}\lambda_f(n s)
\end{equation}
and 
\begin{equation}\label{off-diagonal}
  \mathcal{O}^M \approx \frac{1}{T^2 S}\sum_{s} \sum_{t} \sum_{\psi(t)}\sum_{n \sim N} \sum_{r \sim NS} A(n) \chi(\frac{r}{s}) \sum_{c \ll \sqrt{q} S/T} \frac{1}{c t q}S_\psi (r, n s;c t q)
\end{equation}
which Munshi then works to balance with the appropriate choices
of $S$ and $T$.
(The superscripted $M$
has been included to disambiguate from the closely related expressions
defined in \S\ref{sec:formula-sigma} of this paper.)
In \eqref{off-diagonal} we
sum over moduli $c$ up to the transition range of
the resulting $J$-Bessel function, which we
do not display for notational simplicity.
(For the analogous problem in spectral or $t$-aspects,
the $J$-Bessel function
plays an important analytic role;
cf. forthcoming work of Yongxiao Lin.)

\subsection{Outline of Munshi's method}
We now present a brief outline of Munshi's treatment of $\mathcal{F}^M$ and $\mathcal{O}^M$ (see \cite{2016arXiv160408000M} for details).
\subsubsection{Treatment of  $\mathcal{F}^M$}
\begin{enumerate}
\item Dualize the $n$-sum via the $\GL_3 \times \GL_2$ functional equation.
\item Dualize the $r$-sum via the $\GL_2 \times \GL_1$ functional equation.
\item Sum over $f$ via the Petersson trace formula.
  The diagonal contribution is negligible.  The off-diagonal contribution is a $c$-sum over Kloosterman
  sums of the form $S_\psi(t^2qn,rs;ctq)$ with $c \ll \sqrt{q} T^2$.
\item Factor the Kloosterman sums
  modulo $t$ and modulo $cq$.
  This yields Gauss sums modulo $t$; evaluate them.
  Sum over $\psi$ modulo $t$.
  Factor the remaining Kloosterman sum modulo $c$ and modulo
  $q$.
  The mod $q$ contribution gives a Ramanujan sum
  equal to $-1$.
\item The $n$-sum now oscillates only modulo $c$.  Apply $\GL_3$
  Voronoi and reciprocity.
\item Dualize
  the $c$-sum modulo $r$ via Poisson.
  Only the zero dual frequency contributes.
  It remains to estimate sums of the form
  \begin{equation}\label{keymiddlestep}
   \frac{\sqrt{q}}{T^4} \sum_{t\sim T} \sum_{s \sim S} \sum_{r \sim \sqrt{q}T/S} \sum_{n \sim T^3} A(n) \overline{\chi}(\frac{r s}{t})  S( - \frac{n q}{t}, 1; r s).
  \end{equation}
\item Pull the $n,r$ sums outside and apply Cauchy-Schwarz.
\item Conclude via Poisson in $n$.
\end{enumerate}  
Such a treatment produces the following bound
\begin{equation}\label{moralFbound}
\mathcal{F} \ll N \left[\frac{T}{q^{1/4}S^{1/2}}+\left(\frac{T S}{q^{1/2}}\right)^{1/4}+\textnormal{noise}_\mathcal{F}\right],
\end{equation}
where $\textnormal{noise}_\mathcal{F}$ comes from all of the other technical aspects resulting from working outside of the transition ranges and appropriately setting up the remaining object for each step of the above proof.
\subsubsection{Treatment of $\mathcal{O}^M$}
\begin{enumerate}
\item Factor the Kloosterman sums
  modulo $t$ and $cq$.
  Evaluate
  the sum over $\psi$; this simplifies the Kloosterman
  sums modulo $t$ to additive characters.
  Apply reciprocity.
  One now has oscillations only modulo $cq$.
\item Apply Poisson to the $r$ sum.
  Only the zero frequency contributes non-negligibly
  to the dual sum.
  One is now left with estimating sums of the form
  \begin{equation}\label{mainoffdiag}
   \frac{1}{TS \sqrt{q}} \sum_{t\sim T}  \sum_{s \sim S} \sum_{c \ll \sqrt{q} S/T} \sum_{n\sim N} A(n)  \chi(\frac{t c}{s}) \mathcal{D}(\frac{n s}{ t c};q),
  \end{equation}
  where
  \begin{equation}
    \mathcal{D}(u;q):=\sum_{\substack{b(q) \\ (b(b-1),q)=1}} \overline{\chi}(b-1) e_q((b^{-1} - 1) u)
  \end{equation}
\item Apply Cauchy--Schwarz with the $n$-sum outside.
\item Conclude via Poisson in $n$.
\end{enumerate}
Such a treatment produces the following bound
\begin{equation}\label{moralObound}
\mathcal{O} \ll N \left[\frac{q^{1/4}}{T}+\frac{S}{T}+\textnormal{noise}_\mathcal{O}\right]
\end{equation}
where $\textnormal{noise}_\mathcal{O}$ comes from all of the other technical aspects resulting from working outside of the transition ranges and appropriately setting up the remaining object for each step of the above proof.
\subsubsection{Optimization}
Ignoring the contributions from $\textnormal{noise}_\mathcal{F}$ and  $\textnormal{noise}_\mathcal{O}$ in \eqref{moralFbound} and \eqref{moralObound}, one first restricts $S<q^{1/4}$, sets 
$$
\frac{T}{q^{1/4}S^{1/2}}=\left(\frac{TS}{q^{1/2}}\right)^{1/4}
$$
to get that $S=Tq^{-1/6}$, and then sets 
$$
\frac{T}{q^{1/4}S^{1/2}}=\frac{q^{1/4}}{T}
$$
to get that $T=q^{5/18}$ and $S=q^{2/18}$ which would produce a combined bound of 
\begin{equation}\label{moralbound}
\sum_{n \sim N} A(n)\chi(n) \ll N \left[q^{-1/36}+\textnormal{noise}_{\mathcal{F}+\mathcal{O}}\right].
\end{equation}
Therefore, the best possible bound that one could hope to
achieve is a saving over the convexity bound of size
$q^{-1/36}$. However, due to all of the technical obstacles that
present themselves in the course of the proof, Munshi's
original approach \cite{MR3418527}
produced a saving of $q^{-1/1612}$,
improved in the preprint \cite{2016arXiv160408000M}
to $q^{-1/308}$.

\subsection{Discovering the key identity \eqref{eq:formula-chi}}

After a topics course taught by the first author in the Fall of 2016
and subsequent discussions with the second author in June 2017, the key identity in this paper was discovered hidden within Munshi's work.  Indeed, starting from \eqref{keymiddlestep} in the treatment of $\mathcal{F}^M$, if one were to now apply Voronoi summation in the $n$ sum followed by an application of reciprocity for the resulting additive characters, then one would need to instead analyze sums of the form
\begin{equation} \label{newexplanation}
 \frac{1}{T^2}  \sum_{t\sim T} \sum_{s \sim S} \sum_{r \sim \sqrt{q}T/S} \sum_{n \sim q^{3/2}} \overline{A}(n) \overline{\chi}(\frac{r s}{t})  e_q(-\frac{n t}{r s}).
\end{equation}
Viewing $-t/rs$ as the $u$ in \eqref{eq:formula-chi}, we see that an application of Poisson summation in $r$ returns us to the dual of our original object of interest (from the $h=0$ frequency of the dual) plus a sum which is the ``$\GL_3$ dual'' of $\mathcal{O}^M$ (from the dual non-zero $h$ frequencies) as expressed in \eqref{mainoffdiag}
\begin{equation}\label{newO}
\frac{1}{TS\sqrt{q}}\sum_{t\sim T} \sum_{s \sim S} \sum_{h \ll \sqrt{q}S/T} \sum_{n \sim q^{3/2}} \overline{A}(n) S_{\overline{\chi}}(\frac{h t}{s},n,q).
\end{equation}
By ``$\GL_3$ dual,'' we mean that Voronoi summation in $n$
applied to \eqref{newO} returns one to objects of the form
\eqref{mainoffdiag}. This observation led to the simplification presented in this
paper whereby many of the initial steps of Munshi's argument, as outlined above, are eliminated.

\subsection*{Acknowledgements}
This work was initiated during a visit of PN to RH at The Ohio
State University in June 2017. RH thanks PN for taking the time
to schedule that visit to Columbus on his return to ETH from
MSRI.  RH also thanks the Department of Mathematics at The Ohio
State University for giving him the opportunity to teach a
topics course on R. Munshi's delta method during the Fall 2016
semester. PN thanks The Ohio State University, the STEAM
Factory, and the Erd\H{o}s Institute for their hospitality.
We thank Ritabrata Munshi for encouragement
and Yongxiao Lin for helpful corrections and feedback on an earlier draft.
We thank the referee
for helpful feedback
which has led
to corrections and
improvements to the exposition.

\bibliography{refs}{}
\bibliographystyle{plain}
\end{document}